\documentclass[runningheads]{llncs}

\usepackage{amsmath,amssymb,mathtools}
\usepackage{microtype}
\usepackage[hidelinks]{hyperref}

\DeclareMathOperator{\per}{per}
\DeclareMathOperator{\sgn}{sgn}

\newcommand{\NN}{\mathbb N}
\newcommand{\RR}{\mathbb R}
\newcommand{\CC}{\mathbb C}

\begin{document}

\title{Structural Classes for Chollet's Permanent Conjecture}
\titlerunning{Structural Classes for Chollet's Permanent Conjecture}

\author{Priyanshu Pant\orcidID{0009-0004-1210-0938} \and Ranveer Singh\orcidID{0000-0002-4566-0331}}
\authorrunning{P. Pant and R. Singh}

\institute{
Indian Institute of Technology Indore, India\\
\email{priyanshupant03@gmail.com, ranveer@iiti.ac.in}
}

\maketitle
\pagestyle{empty}
\thispagestyle{empty}

\begin{abstract}
In 1982, Chollet conjectured that
\(
\per(A\circ B)\le \per(A)\per(B)
\)
for Hermitian positive semidefinite matrices $A,B$, where $\circ$ denotes the Hadamard (entrywise) product. In this paper, we study natural structural classes for which the conjecture holds. We first show that a stronger inequality holds for a broad class of matrices with bipartite support. We then introduce a simple way of joining positive semidefinite matrices
and give conditions under which Chollet's inequality is preserved under
this operation. For graph Laplacians, this operation corresponds to vertex coalescence and gives larger structured graph families satisfying Chollet's inequality from simpler graph classes.
\end{abstract}

\section{Introduction}

The \emph{determinant} and \emph{permanent} of an $n\times n$ matrix
$A=(a_{ij})$ are
\[
\det(A)=\sum_{\sigma\in S_n}\sgn(\sigma)\prod_{i=1}^n a_{i,\sigma(i)},
\qquad
\per(A)=\sum_{\sigma\in S_n}\prod_{i=1}^n a_{i,\sigma(i)}.
\]
Although the two expressions differ only by signs, their behaviour is
very different. The determinant can be computed using $O(n^3)$ arithmetic
operations by Gaussian elimination, whereas computing the
permanent is $\#P$-complete even for
$(0,1)$-matrices~\cite{valiant1979complexity}.
Moreover, many basic determinant properties do not extend to permanents. For example, the permanent is not invariant under elementary row operations and is not multiplicative in general~\cite{Minc1978}. As a result, proving nontrivial inequalities for permanents is often difficult. Permanents of positive semidefinite matrices have also received attention in connection with quantum information and linear optics~\cite{ChakhmakhchyanCerfGarciaPatron2017}.

We write $A\succeq0$ when $A$ is Hermitian positive semidefinite. For a Hermitian matrix $A$, the permanent terms indexed by $\sigma$ and $\sigma^{-1}$ are complex conjugates, so $\per(A)$ is real. If $A\succeq0$, Lieb's inequality gives $\per(A)\ge\prod_i a_{ii}\ge0$~\cite{Lieb1966}. For
$A,B\in\CC^{n\times n}$, their Hadamard product is
$(A\circ B)_{ij}=a_{ij}b_{ij}$. By the Schur product theorem,
$A\circ B\succeq0$ whenever $A,B\succeq0$~\cite{hornjohnson2012matrix}.
For determinants, Oppenheim's and Hadamard's inequalities give
\[
\det(A\circ B)\ge \det(A)\prod_{i=1}^n b_{ii}
\ge \det(A)\det(B).
\]
Chollet~\cite{Chollet1982} proposed the corresponding permanent
inequality.

\begin{conjecture}\label{conj:chollet}
If $A,B\succeq0$ are Hermitian positive semidefinite matrices, then
\[
\per(A\circ B)\le\per(A)\per(B).
\]
\end{conjecture}

For $n\le3$, the conjecture was proved by
Gregorac--Hentzel~\cite{GregoracHentzel1987}. For $n=4$, it is known for
real PSD matrices by Hutchinson~\cite{Hutchinson2021} and for general PSD
matrices by Rodtes~\cite{Rodtes01112024}. The conjecture remains open in
general.
Several works study the conjecture for correlation matrices, since matrices with positive diagonal can be scaled to this case; see
Grone--Pierce~\cite{GronePierce1988},
Zhang~\cite{Zhang1989,Zhang2013,Zhang2016}, and
Sa-Nguansin--Rodtes~\cite{SaNguansinRodtes2025}.
These works suggest that the structure of the matrix also matters, not only
its size.
Chollet~\cite{Chollet1982} observed that it is enough to prove
\begin{equation}\label{eq:self-complex}
\per(A\circ\overline A)\le |\per(A)|^2
\end{equation}
for every Hermitian PSD matrix $A$, where $\overline A$
denotes its entrywise complex conjugate.
For real symmetric $A$, this becomes
$\per(A\circ A)\le\per(A)^2$.
Bapat and Sunder~\cite{BapatSunder1985} conjectured the stronger inequality
\begin{equation}\label{eq:bapat-sunder}
\per(A\circ B)\le\per(A)\prod_{i=1}^n b_{ii},
\end{equation}
for Hermitian positive semidefinite matrices $A$ and $B=(b_{ij})$.
Since Lieb's permanent inequality gives
$\per(B)\ge\prod_i b_{ii}$, Equation~\eqref{eq:bapat-sunder} would imply
Chollet's conjecture. However, the Bapat--Sunder conjecture was disproved by
Drury~\cite{Drury2016}.

These observations motivate the search for structural conditions under which
Chollet's inequality, or even the stronger Bapat--Sunder inequality, holds.
Graph Laplacians form a particularly natural structured class. They are real
positive semidefinite matrices with nonpositive off-diagonal entries, their
zero pattern is determined by the underlying graph, and their permanent terms
admit a direct interpretation in terms of cycle covers.

In this paper, we prove Chollet's inequality for several structured matrix
classes and show how these classes can be combined to produce larger
families. We first prove a stronger inequality for real symmetric matrices
with nonnegative diagonal, nonpositive off-diagonal entries, and bipartite
support (Theorem~\ref{thm:bipartite-strong}). We then introduce a one-point
sum of positive semidefinite matrices and show that the inequality is
preserved under repeated composition
(Theorem~\ref{thm:hereditary-closure}). For graph Laplacians, this operation
becomes vertex coalescence, which reduces the problem to understanding the
individual parts of the graph. We establish the required inequalities for
several basic graph classes, including bipartite graphs, cycles, and complete
graphs, and then combine them through cut vertices to obtain larger graph
families
(Theorem~\ref{thm:cactus} and Corollary~\ref{cor:mixed-blocks}).

The paper is organized as follows. Section~\ref{sec:prelim} gives the
basic notation and inequalities. Section~\ref{sec:bipartite} proves the
result for matrices with bipartite support. Section~\ref{sec:one-point}
studies one-point sums and the Chollet defect. Section~\ref{sec:graphs}
applies these results to graph Laplacians. 
The longer proofs are given in Appendices~\ref{app:strong-closure} and~\ref{app:clique-calculations}.

\section{Preliminaries}\label{sec:prelim}

We fix notation and record the basic permanent inequalities used later.
For $n\in\NN$, write $[n]=\{1,\ldots,n\}$. For a complex number
$z=x+iy$, its complex conjugate is
\(
\overline z=x-iy.
\)
For a matrix $A=(a_{ij})\in\CC^{n\times n}$, let
\(
\overline A=(\overline{a_{ij}})
\)
denote its entrywise complex conjugate, and let
\(
A^\dagger=\overline A^{\,T}
\)
denote its conjugate transpose.
For $S\subseteq[n]$, let $A[S]$ be the principal submatrix of $A$
indexed by $S$, and let $A(i)$ be the principal submatrix obtained by
deleting row and column $i$. Let $E_{ij}$ denote the matrix whose
$(i,j)$-entry is $1$ and whose remaining entries are all $0$.
We write $I_m$ and $J_m$ for the $m\times m$ identity and all-ones
matrices, respectively.
For square matrices $A$ and $B$, their block diagonal direct sum is
\[
A\oplus B=
\begin{pmatrix}
A&0\\
0&B
\end{pmatrix}.
\]
A real matrix is a \emph{$Z$-matrix} if its off-diagonal entries are
nonpositive. For a real symmetric matrix $A=(a_{ij})$, its
\emph{support graph} has vertex set $[n]$ and an edge $\{i,j\}$ whenever
$i\ne j$ and $a_{ij}\ne0$.
All graphs in this paper are simple. For a graph $G$, let $A_G$ be its
adjacency matrix, $D_G$ its degree matrix, and
\[
L_G=D_G-A_G
\]
its Laplacian. For $v\in V(G)$, let $L_G(v)$ be the principal submatrix
obtained by deleting the row and column indexed by $v$.

We will use the following form of Lieb's inequality throughout the paper~\cite{Lieb1966}.

\begin{lemma}\label{lem:lieb}
If
\(
A=
\begin{pmatrix}
B&C\\
C^\dagger&D
\end{pmatrix}
\succeq0,
\)
then
\(
\per(A)\ge\per(B)\per(D).
\)
\end{lemma}

\begin{corollary}\label{cor:lieb-principal}
If $A=(a_{ij})\succeq0$, then for every $i$,
\[
\per(A)\ge a_{ii}\per(A(i)) \text{, and }
\per(A)\ge\prod_{i=1}^n a_{ii}.
\]
\end{corollary}

\begin{proof}
The first inequality follows from Lemma~\ref{lem:lieb} by taking the
$i$th diagonal entry as a $1\times1$ block. Repeating the same argument
on principal submatrices gives the product of the diagonal entries.
\end{proof}

The usual Cauchy--Schwarz inequality gives the following reduction.

\begin{lemma}\label{lem:self-to-pair}
For complex matrices $A,B$ of the same order,
\[
|\per(A\circ B)|^2
\le
\per(A\circ\overline A)\per(B\circ\overline B).
\]
\end{lemma}

\begin{proof}
For $\sigma\in S_n$, let
$x_\sigma=\prod_{i=1}^n a_{i,\sigma(i)}$ and
$y_\sigma=\prod_{i=1}^n b_{i,\sigma(i)}$. Then
\[
\per(A\circ B)=\sum_{\sigma\in S_n}x_\sigma y_\sigma.
\]
By the Cauchy--Schwarz inequality,
\[
|\per(A\circ B)|^2
\le
\left(\sum_{\sigma\in S_n}|x_\sigma|^2\right)
\left(\sum_{\sigma\in S_n}|y_\sigma|^2\right).
\]
Since
\[
\sum_{\sigma\in S_n}|x_\sigma|^2
=\per(A\circ\overline A)
\qquad\text{and}\qquad
\sum_{\sigma\in S_n}|y_\sigma|^2
=\per(B\circ\overline B),
\]
the result follows.
\end{proof}

Thus if two PSD matrices satisfy Equation~\eqref{eq:self-complex},
Lemma~\ref{lem:self-to-pair} gives
\[
|\per(A\circ B)|\le\per(A)\per(B).
\]
Since $A\circ B\succeq0$, its permanent is nonnegative, and Chollet's
inequality follows.

\section{Matrices with bipartite support}\label{sec:bipartite}

We first prove a stronger inequality for real symmetric matrices with nonnegative diagonal, nonpositive off-diagonal entries, and bipartite support. For such a matrix $A$, every nonzero permutation term is nonnegative. We combine this fact with the bound $|b_{ij}|\le\sqrt{b_{ii}b_{jj}}$ for a positive semidefinite matrix $B=(b_{ij})$.

\begin{theorem}\label{thm:bipartite-strong}
Let $A=(a_{ij})\in\RR^{n\times n}$ be a symmetric $Z$-matrix with
nonnegative diagonal and bipartite support. For every Hermitian
positive semidefinite matrix $B=(b_{ij})$,
\[
|\per(A\circ B)|
\le
\per(A)\prod_{i=1}^n b_{ii}.
\]
\end{theorem}

\begin{proof}
For $\sigma\in S_n$, set
\[
x_\sigma=\prod_i a_{i,\sigma(i)}.
\]
We first show that every nonzero $x_\sigma$ is nonnegative. Consider a
nontrivial cycle of $\sigma$. A $2$-cycle contributes two negative
off-diagonal entries and therefore has positive product. A cycle of
length at least $3$ uses edges of the support graph, so its length is
even because the support graph is bipartite. Thus every nontrivial cycle
contributes an even number of negative factors. Fixed points contribute
nonnegative diagonal entries. Hence $x_\sigma\ge0$.

Now put
\(
y_\sigma=\prod_i b_{i,\sigma(i)}.
\)
Since $B\succeq0$,
\(
|b_{ij}|\le\sqrt{b_{ii}b_{jj}},
\)
and therefore
\[
|y_\sigma|
\le
\prod_i\sqrt{b_{ii}b_{\sigma(i),\sigma(i)}}
=
\prod_i b_{ii}.
\]
It follows that
\[
|\per(A\circ B)|
=
\left|\sum_{\sigma\in S_n}x_\sigma y_\sigma\right|
\le
\sum_{\sigma\in S_n}x_\sigma|y_\sigma|
\le
\per(A)\prod_i b_{ii}.
\]
\end{proof}

Theorem~\ref{thm:bipartite-strong} gives the Bapat--Sunder inequality~\eqref{eq:bapat-sunder} under the stated sign and support conditions, without requiring $A$ to be positive semidefinite. Since $B\succeq0$, Corollary~\ref{cor:lieb-principal} gives $\prod_i b_{ii}\le\per(B)$, so Chollet's inequality also follows. When $A$ is positive semidefinite, this proves Conjecture~\ref{conj:chollet} for this class. The class in Theorem~\ref{thm:bipartite-strong} contains all Laplacians of bipartite graphs with positive edge weights, but does not require the diagonal entries to equal the weighted degrees. We next record two immediate consequences, first for
graph Laplacians and then for principal submatrices.

\begin{corollary}\label{cor:bipartite-graphs}
If $G$ is bipartite and $H$ is any graph on the same vertex set, then
\[
\per(L_G\circ L_H)
\le
\per(L_G)\prod_{v\in V(H)}\deg_H(v)
\le
\per(L_G)\per(L_H).
\]
\end{corollary}

\begin{proof}
Apply Theorem~\ref{thm:bipartite-strong} with $A=L_G$ and $B=L_H$.
The second inequality follows from Corollary~\ref{cor:lieb-principal}.
\end{proof}

\begin{corollary}\label{cor:bipartite-principal}
Let $A\succeq0$ be a real symmetric $Z$-matrix with bipartite support.
Every principal submatrix $A[S]$ satisfies
\[
\per(A[S]\circ A[S])
\le
\per(A[S])\prod_{i\in S}a_{ii}
\le
\per(A[S])^2.
\]
\end{corollary}

\begin{proof}
Every principal submatrix of $A$ is again PSD and has bipartite support.
Apply Theorem~\ref{thm:bipartite-strong} to $A[S]$ and then use
Corollary~\ref{cor:lieb-principal}.
\end{proof}

\section{One-Point Sums of Positive Semidefinite Matrices}
\label{sec:one-point}

We now introduce a simple way of joining two positive semidefinite
matrices. We will derive formulas for the permanent under this operation
and use them to show that Chollet's inequality is preserved under
suitable compositions. Let
\[
A_1=
\begin{pmatrix}
B_1&c_1\\
c_1^\dagger&d_1
\end{pmatrix}
\succeq0
\qquad\text{and}\qquad
A_2=
\begin{pmatrix}
B_2&c_2\\
c_2^\dagger&d_2
\end{pmatrix}
\succeq0.
\]
In each matrix, we distinguish the index corresponding to the
bottom-right entry. We join these two indices and define
\[
A_1\star A_2=
\begin{pmatrix}
B_1&c_1&0\\
c_1^\dagger&d_1+d_2&c_2^\dagger\\
0&c_2&B_2
\end{pmatrix}.
\]
We call $A_1\star A_2$ the \emph{one-point sum} of $A_1$ and $A_2$.
It is positive semidefinite because
\[
A_1\star A_2
=
\begin{pmatrix}
B_1&c_1&0\\
c_1^\dagger&d_1&0\\
0&0&0
\end{pmatrix}
+
\begin{pmatrix}
0&0&0\\
0&d_2&c_2^\dagger\\
0&c_2&B_2
\end{pmatrix},
\]
and both summands are positive semidefinite.
For a Hermitian positive semidefinite matrix $A$, define its
\emph{Chollet defect} by
\[
\Delta(A)=\per(A)^2-\per(A\circ\overline A).
\]
Thus $\Delta(A)$ measures the difference between the two sides of the
self-Chollet inequality~\eqref{eq:self-complex}. Since $\per(A)\ge0$ for
$A\succeq0$, the self-Chollet inequality holds exactly when
$\Delta(A)\ge0$.

We first derive the permanent formulas needed for the one-point sum.

\begin{proposition}\label{prop:one-point-formulas}
For $A=A_1\star A_2$,
\begin{equation}\label{eq:one-point-per}
\per(A)=\per(A_1)\per(B_2)+\per(B_1)\per(A_2),
\end{equation}
and
\begin{align}
\per(A\circ\overline A)
={}&
\per(A_1\circ\overline{A_1})\per(B_2\circ\overline{B_2})
+\per(B_1\circ\overline{B_1})\per(A_2\circ\overline{A_2})
\notag\\
&+
2d_1d_2
\per(B_1\circ\overline{B_1})
\per(B_2\circ\overline{B_2}).
\label{eq:one-point-had}
\end{align}
\end{proposition}

\begin{proof}
Let $v$ be the common distinguished index. Consider a nonzero permanent
monomial of $A$. Since the blocks connecting $B_1$ and $B_2$ are zero, no permutation cycle can contain vertices from both sides without
passing through $v$. Hence the cycle containing $v$ lies entirely on the
$A_1$ side or entirely on the $A_2$ side, and every other cycle lies
entirely inside $B_1$ or $B_2$.
If the cycle containing $v$ lies on the $A_1$ side, then the vertices on
the other side contribute a permanent monomial of $B_2$. Similarly, if
it lies on the $A_2$ side, the other side contributes a permanent monomial
of $B_1$. When $v$ is fixed, its diagonal contribution is $d_1+d_2$, and
the two terms $d_1$ and $d_2$ belong to these two cases. Therefore
\[
\per(A)
=
\per(A_1)\per(B_2)+\per(B_1)\per(A_2).
\]
The same argument applies to $A\circ\overline A$. The distinguished
diagonal entry is
\[
(d_1+d_2)^2=d_1^2+d_2^2+2d_1d_2.
\]
The terms whose cycle containing $v$ lies on the $A_1$ side, using $d_1^2$ when $v$ is fixed, sum to
\[
\per(A_1\circ\overline{A_1})\per(B_2\circ\overline{B_2}).
\]
Similarly, the terms on the $A_2$ side, using $d_2^2$ when $v$ is fixed, sum to
\[
\per(B_1\circ\overline{B_1})\per(A_2\circ\overline{A_2}).
\]
The remaining term $2d_1d_2$ occurs only when the permutation fixes $v$. The remaining vertices contribute independently on the two sides, giving
\[
2d_1d_2
\per(B_1\circ\overline{B_1})
\per(B_2\circ\overline{B_2}).
\]
Combining the three contributions gives
Equation~\eqref{eq:one-point-had}.
\end{proof}

The permanent formulas above allow us to express the Chollet defect of a
one-point sum in terms of the two matrices and their distinguished
principal submatrices.

\begin{theorem}\label{thm:defect-decomposition}
The defect of the one-point sum satisfies
\begin{align}
\Delta(A_1\star A_2)
={}&
\per(B_2)^2\Delta(A_1)
+\per(A_1\circ\overline{A_1})\Delta(B_2)
+\per(B_1)^2\Delta(A_2)
\notag\\
&+
\per(A_2\circ\overline{A_2})\Delta(B_1)
\notag\\
&+
2\per(B_1)\per(B_2)
\bigl(
\per(A_1)\per(A_2)-d_1d_2\per(B_1)\per(B_2)
\bigr)
\notag\\
&+
2d_1d_2
\bigl(
\per(B_1)^2\per(B_2)^2
-\per(B_1\circ\overline{B_1})
 \per(B_2\circ\overline{B_2})
\bigr).
\label{eq:defect-decomposition}
\end{align}
\end{theorem}
\begin{proof}
Squaring Equation~\eqref{eq:one-point-per} and subtracting Equation~\eqref{eq:one-point-had} gives
\[
\begin{aligned}
\Delta(A_1\star A_2)
={}&
\per(A_1)^2\per(B_2)^2
-\per(A_1\circ\overline{A_1})
 \per(B_2\circ\overline{B_2})\\
&+
\per(B_1)^2\per(A_2)^2
-\per(B_1\circ\overline{B_1})
 \per(A_2\circ\overline{A_2})\\
&+
2\per(A_1)\per(A_2)\per(B_1)\per(B_2)\\
&-
2d_1d_2
\per(B_1\circ\overline{B_1})
\per(B_2\circ\overline{B_2}).
\end{aligned}
\]
For $\{i,j\}=\{1,2\}$, we have
\[
\begin{aligned}
&\per(A_i)^2\per(B_j)^2
-\per(A_i\circ\overline{A_i})\per(B_j\circ\overline{B_j})\\
&\qquad=
\per(B_j)^2\Delta(A_i)
+\per(A_i\circ\overline{A_i})\Delta(B_j).
\end{aligned}
\]
Applying this identity to the first two differences gives the first four terms of Equation~\eqref{eq:defect-decomposition}. For the remaining terms, add and subtract $2d_1d_2\per(B_1)^2\per(B_2)^2$ and factor to obtain the last two terms.
\end{proof}

The formula above shows that the defect of $A_1\star A_2$ is determined
by the defects of $A_1$, $A_2$, $B_1$, and $B_2$, together with two
additional terms. We now show that these terms are also nonnegative when
the four defects are nonnegative.

\begin{corollary}\label{cor:defect-closure}
If
\(
\Delta(A_1), \Delta(B_1), \Delta(A_2), \Delta(B_2)\ge0,
\)
then
\(
\Delta(A_1\star A_2)\ge0.
\)
Moreover,
\begin{equation}\label{eq:defect-lower}
\Delta(A_1\star A_2)
\ge
\per(B_2)^2\Delta(A_1)+\per(B_1)^2\Delta(A_2).
\end{equation}
\end{corollary}

\begin{proof}
By Corollary~\ref{cor:lieb-principal},
$\per(A_i)\ge d_i\per(B_i)$ for $i=1,2$. Hence
\[
\per(A_1)\per(A_2)
\ge
d_1d_2\per(B_1)\per(B_2).
\]
Also, since $\Delta(B_i)\ge0$,
\(
\per(B_i\circ\overline{B_i})\le\per(B_i)^2.
\)
Finally, $A_i\circ\overline{A_i}\succeq0$ by the Schur product theorem, and
therefore
\(
\per(A_i\circ\overline{A_i})\ge0.
\)
Since $B_i\succeq0$ and $d_i\ge0$, we also have $\per(B_i)\ge0$.
Thus every term on the right-hand side of
Equation~\eqref{eq:defect-decomposition} is nonnegative, which gives
\[
\Delta(A_1\star A_2)\ge0.
\]
Keeping only the first and third terms gives
Equation~\eqref{eq:defect-lower}.
\end{proof}

To apply Corollary~\ref{cor:defect-closure} repeatedly, we also need the
same inequality for principal submatrices. We call a Hermitian positive
semidefinite matrix $A$ \emph{hereditarily Chollet} if
\(
\Delta(A[S])\ge0
\)
for every principal submatrix $A[S]$.
We first establish a closure property under increasing diagonal entries.

\begin{proposition}\label{prop:diagonal-loading}
Let $A\succeq0$ be Hermitian. If $\Delta(A)\ge0$ and
$\Delta(A(i))\ge0$, then
\[
\Delta(A+\alpha E_{ii})\ge0
\]
for every $\alpha\in\RR_{\ge0}$. In particular, if $A$ is hereditarily Chollet
and $D\succeq0$ is diagonal, then $A+D$ is hereditarily Chollet.
\end{proposition}

\begin{proof}
Put $A'=A+\alpha E_{ii}$. By multilinearity,
$\per(A')=\per(A)+\alpha\per(A(i))$. Also,
\[
\per(A'\circ\overline{A'})
=
\per(A\circ\overline A)
+
(2a_{ii}\alpha+\alpha^2)
\per(A(i)\circ\overline{A(i)}).
\]
Using $\Delta(A)\ge0$ and $\Delta(A(i))\ge0$ gives
\[
\per(A'\circ\overline{A'})
\le
\per(A)^2+
(2a_{ii}\alpha+\alpha^2)\per(A(i))^2.
\]
By Corollary~\ref{cor:lieb-principal},
$\per(A)\ge a_{ii}\per(A(i))$. Hence
\[
\begin{aligned}
\per(A'\circ\overline{A'})
&\le
\per(A)^2
+2\alpha\per(A)\per(A(i))
+\alpha^2\per(A(i))^2\\
&=
\bigl(\per(A)+\alpha\per(A(i))\bigr)^2
=
\per(A')^2.
\end{aligned}
\]
Thus $\Delta(A')\ge0$.
Now suppose that $A$ is hereditarily Chollet. Add the diagonal entries
of $D$ one at a time. For $A'=A+\alpha E_{ii}$, let $C=A[S]$ be a
principal submatrix. If $i\notin S$, then $A'[S]=C$. If $i\in S$ and
$|S|\ge2$, then both $C$ and $C(i)$ have nonnegative defect, so the
first part applies to $C$. If $S=\{i\}$, the inequality is immediate.
Hence $A'$ is hereditarily Chollet. Repeating this for every diagonal
entry of $D$ proves the result.
\end{proof}

This hereditary property is preserved under one-point sums.

\begin{theorem}\label{thm:hereditary-closure}
If $A_1$ and $A_2$ are hereditarily Chollet, then
$A_1\star A_2$ is hereditarily Chollet.
\end{theorem}

\begin{proof}
Let $C$ be a principal submatrix of $A_1\star A_2$.
Suppose first that $C$ does not contain the common distinguished index.
If $C$ lies entirely on one side, then the result follows from the
hereditary Chollet property of $A_1$ or $A_2$. Otherwise,
\[
C=C_1\oplus C_2,
\]
where $C_i$ is a principal submatrix of $B_i$. Hence
\[
\per(C\circ\overline C)
=
\per(C_1\circ\overline{C_1})
\per(C_2\circ\overline{C_2})
\le
\per(C_1)^2\per(C_2)^2
=
\per(C)^2.
\]
Now suppose that $C$ contains the common distinguished index. If $C$
consists only of this index, the result is immediate. If $C$ contains
vertices from both sides, then
\[
C=C_1\star C_2,
\]
where $C_i$ is a principal submatrix of $A_i$ containing its
distinguished index. Deleting the distinguished index from $C_i$ gives a principal submatrix of $A_i$ for $i=1,2$. Hence all four defects are nonnegative,
and Corollary~\ref{cor:defect-closure} gives $\Delta(C)\ge0$.
If $C$ contains vertices from only one side, then it is obtained from a
principal submatrix of $A_1$ or $A_2$ by increasing its distinguished
diagonal entry. Proposition~\ref{prop:diagonal-loading} therefore gives
$\Delta(C)\ge0$.
Hence every principal submatrix of $A_1\star A_2$ has nonnegative
defect.
\end{proof}

For real symmetric matrices, we also consider the stronger inequality
\(
\per(A\circ A)\le\per(A)\prod_i a_{ii}.
\)
We call a real symmetric positive semidefinite matrix $A$
\emph{hereditarily strong} if every principal submatrix
$C=(c_{ij})$ satisfies
\[
\per(C\circ C)\le\per(C)\prod_i c_{ii}.
\]

\begin{theorem}\label{thm:hereditary-strong}
If $A_1$ and $A_2$ are hereditarily strong, then
$A_1\star A_2$ is hereditarily strong.
\end{theorem}

\begin{proof}
Proposition~\ref{prop:one-point-formulas}, together with the strong
inequalities for $A_i$ and $B_i$ and
Corollary~\ref{cor:lieb-principal}, gives the strong inequality for
$A_1\star A_2$. The same argument applies to principal submatrices;
the remaining one-sided case is handled by increasing the distinguished
diagonal entry.
The details are given in Appendix~\ref{app:strong-closure}.
\end{proof}

\section{Graph Laplacians}\label{sec:graphs}

We now apply the one-point-sum results to graph Laplacians. In this
setting, the one-point sum corresponds to joining two graphs at one
vertex.

\subsection{Vertex coalescence and blocks}

Let $G_1$ and $G_2$ be vertex-disjoint graphs with distinguished
vertices $v_i\in V(G_i)$. Their \emph{vertex coalescence}
$G=G_1\cdot G_2$ is obtained by identifying $v_1$ and $v_2$ into one
vertex $v$.
Put $U_i=V(G_i)\setminus\{v_i\}$, and let $c_i$ be the column of
$L_{G_i}$ indexed by $v_i$, restricted to $U_i$. Ordering the vertices
as $U_1,v,U_2$, we have
\[
L_G=
\begin{pmatrix}
L_{G_1}(v_1)&c_1&0\\
c_1^\top&
\deg_{G_1}(v_1)+\deg_{G_2}(v_2)&
c_2^\top\\
0&c_2&L_{G_2}(v_2)
\end{pmatrix}.
\]
Thus $L_G=L_{G_1}\star L_{G_2}$.
For a graph $G$, write
\[
\Delta(G)=\Delta(L_G)
=\per(L_G)^2-\per(L_G\circ L_G).
\]

\begin{corollary}\label{cor:vertex-coalescence}
Suppose $G=G_1\cdot G_2$ and
$\Delta(L_{G_i})\ge0$ and $\Delta(L_{G_i}(v_i))\ge0$ for $i=1,2$.
Then $\Delta(G)\ge0$.
\end{corollary}

\begin{proof}
Since $L_G=L_{G_1}\star L_{G_2}$, the result follows from
Corollary~\ref{cor:defect-closure}.
\end{proof}

We next describe the corresponding decomposition of a connected graph.
A \emph{cut vertex} is a vertex whose deletion disconnects the graph.
A graph with at least three vertices is \emph{$2$-connected} if it is
connected and has no cut vertex. The \emph{blocks} of a connected graph
$G$ are its maximal $2$-connected subgraphs together with its bridge
edges, where every bridge is regarded as a copy of $K_2$. Let
$\mathcal B(G)$ denote the set of blocks.
The \emph{block-cut tree} of a connected graph \(G\) is the bipartite graph whose vertices are the blocks and cut vertices of \(G\), with a block \(B\) adjacent to a cut vertex \(v\) whenever \(v\in V(B)\). This graph is a tree.

\begin{theorem}[Block closure]\label{thm:block-closure}
If $L_B$ is hereditarily Chollet for every block $B$ of a connected
graph $G$, then $L_G$ is hereditarily Chollet.
\end{theorem}

\begin{proof}
If $G=K_1$, then $L_G=(0)$ is hereditarily Chollet.
We may therefore assume that $G$ has at least two vertices.
We proceed by induction on the number of blocks. The result is immediate
when $G$ has one block.
Let $B$ be a leaf block of the block-cut tree and let $v$ be its unique
cut vertex. Delete the vertices of $B\setminus\{v\}$ and call the
remaining graph $H$. Then $G=H\cdot B$. By induction, $L_H$ is
hereditarily Chollet, while $L_B$ is hereditarily Chollet by assumption.
Theorem~\ref{thm:hereditary-closure} gives the result.
\end{proof}

The previous result gives a little more.

\begin{theorem}\label{thm:defect-superadditivity}
If $L_B$ is hereditarily Chollet for every block $B$ of a connected
graph $G$, then
\[
\Delta(G)\ge\sum_{B\in\mathcal B(G)}\Delta(B).
\]
\end{theorem}

\begin{proof}
The inequality is an equality when $G=K_1$ or $G$ has one block. We proceed by induction on the number of blocks and assume that $G$ has at least two blocks. Let $B$ be a leaf block,
let $v$ be its unique cut vertex, and let $H$ be obtained by deleting
$B\setminus\{v\}$. Then $G=H\cdot B$, and
Equation~\eqref{eq:defect-lower} gives
\[
\Delta(G)
\ge
\per(L_B(v))^2\Delta(H)
+
\per(L_H(v))^2\Delta(B).
\]
For every connected nontrivial graph $X$ and $x\in V(X)$,
Corollary~\ref{cor:lieb-principal} gives
$\per(L_X(x))\ge\prod_{u\ne x}\deg_X(u)\ge1$. Hence
$\Delta(G)\ge\Delta(H)+\Delta(B)$. Applying the induction hypothesis to
$H$ proves the result.
\end{proof}

\subsection{Structured Graph Classes and Their Extensions}

Corollary~\ref{cor:bipartite-principal} gives the stronger inequality for every principal submatrix of a bipartite graph Laplacian. We now consider cycles and complete graphs.
A \emph{matching} in a graph is a set of edges no two of which share a
vertex. For a graph $H$, define
\[
M_H(z)=\sum_{k\ge0}m_k(H)z^k,
\]
where $m_k(H)$ is the number of matchings with $k$ edges.

\begin{theorem}[Cycles]\label{thm:cycles}
For every cycle $C_n$ with $n\ge3$,
\[
\per(L_{C_n}\circ L_{C_n})
\le
\per(L_{C_n})\prod_{v\in V(C_n)}\deg_{C_n}(v).
\]
Every principal submatrix of $L_{C_n}$ satisfies the corresponding
inequality with the product of its diagonal entries.
\end{theorem}

\begin{proof}
Write $L_{C_n}=2B$, where $B=I_n-\frac12A_{C_n}$. Since every diagonal
entry of $B$ is $1$, it is enough to prove
$\per(B\circ B)\le\per(B)$.
The only nontrivial permutation cycles supported by $C_n$ are
transpositions corresponding to edges and the two directed cycles using
all $n$ vertices. A set of disjoint transpositions is exactly a
matching, and each matching edge contributes $(-1/2)^2=1/4$. Hence
\[
\per(B)
=
M_{C_n}\!\left(\frac14\right)
+
2\left(-\frac12\right)^n.
\]
For $B\circ B$, each matching edge contributes $1/16$, while the two
directed cycles contribute $(1/4)^n$. Thus
\[
\per(B\circ B)
=
M_{C_n}\!\left(\frac1{16}\right)
+
2\left(\frac14\right)^n.
\]
If $n$ is even, both the matching contribution and the directed-cycle
contribution in the first expression are at least the corresponding
terms in the second. Hence $\per(B)\ge\per(B\circ B)$.
Now suppose that $n$ is odd. The one-edge matchings alone give
\[
M_{C_n}\!\left(\frac14\right)
-
M_{C_n}\!\left(\frac1{16}\right)
\ge
n\left(\frac14-\frac1{16}\right)
=
\frac{3n}{16}.
\]
Since $\frac{3n}{16}\ge2^{1-n}+2^{1-2n}$ for $n\ge3$, we obtain
\[
\begin{aligned}
\per(B)-\per(B\circ B)
&=
M_{C_n}\!\left(\frac14\right)
-
M_{C_n}\!\left(\frac1{16}\right)-2^{1-n}-2^{1-2n}
\ge0.
\end{aligned}
\]

Finally, every proper principal submatrix of $L_{C_n}$ has support equal
to a disjoint union of paths and isolated vertices. Its support is
therefore bipartite, so Corollary~\ref{cor:bipartite-principal} gives
the stronger inequality for every proper principal submatrix.
\end{proof}

Thus $L_{C_n}$ is hereditarily strong for every $n\ge3$.
A connected graph is a \emph{cactus graph} if every block is either an
edge or a cycle.

\begin{theorem}[Cactus graphs]\label{thm:cactus}
If $G$ is a connected cactus graph, then $L_G$ is hereditarily strong.
In particular,
\[
\per(L_G\circ L_G)
\le
\per(L_G)\prod_{v\in V(G)}\deg_G(v)
\le
\per(L_G)^2.
\]
\end{theorem}

\begin{proof}
Every edge is bipartite, so its Laplacian is hereditarily strong by
Corollary~\ref{cor:bipartite-principal}. Every other block of $G$ is a
cycle and is therefore hereditarily strong by
Theorem~\ref{thm:cycles}. Applying
Theorem~\ref{thm:hereditary-strong} along the block-cut tree shows that
$L_G$ is hereditarily strong. The two inequalities then follow from the
definition and Corollary~\ref{cor:lieb-principal}.
\end{proof}

More generally, $L_G$ is hereditarily strong whenever every block of $G$ is bipartite or a cycle.

The self-Chollet inequality for complete-graph Laplacians follows from Sa-Nguansin--Rodtes~\cite{SaNguansinRodtes2025} and the known cases of order at most four~\cite{Rodtes01112024}. We obtain the hereditary property by adding a nonnegative diagonal matrix.

\begin{theorem}[Complete graphs]\label{thm:complete-graphs}
For every $n\ge1$, the Laplacian $L_{K_n}$ is hereditarily Chollet.
\end{theorem}

\begin{proof}
We proceed by induction on $n$. The case $n=1$ is immediate.
For $n\ge2$, the calculation in Appendix~\ref{app:clique-calculations} gives
$\Delta(L_{K_n})\ge0$.
Now let $S\subsetneq[n]$ with $s=|S|\ge1$. Then
\(
L_{K_n}[S]
=
nI_s-J_s
=
L_{K_s}+(n-s)I_s.
\)
By induction, $L_{K_s}$ is hereditarily Chollet.
Proposition~\ref{prop:diagonal-loading} therefore shows that
$L_{K_n}[S]$ is hereditarily Chollet. Hence $L_{K_n}$ is hereditarily
Chollet.
\end{proof}

Combining Theorem~\ref{thm:block-closure} with the hereditary Chollet property of bipartite graphs (Corollary~\ref{cor:bipartite-principal}), cycles (Theorem~\ref{thm:cycles}), and complete graphs (Theorem~\ref{thm:complete-graphs}) gives the following.

\begin{corollary}\label{cor:mixed-blocks}
If every block of a connected graph $G$ is bipartite, a cycle, or a
complete graph, then $L_G$ is hereditarily Chollet.
\end{corollary}

The conclusion extends to disconnected graphs, since their Laplacians are direct sums over connected components and permanents multiply over direct sums. The class in Corollary~\ref{cor:mixed-blocks} is recognizable in $O(|V(G)|+|E(G)|)$ time: compute the blocks by depth-first search~\cite{Tarjan1972}, then test each for bipartiteness or whether it is a cycle or complete.

\section{Concluding Remarks}\label{sec:discussion}

We have given conditions preserving the self-Chollet inequality~\eqref{eq:self-complex} under one-point sums of positive semidefinite matrices. We also proved the stronger Bapat--Sunder inequality~\eqref{eq:bapat-sunder} for matrices with bipartite support. Consequently, Chollet's inequality holds for graph Laplacians whose blocks are bipartite, cycles, or complete graphs. If every block is bipartite or a cycle, the stronger inequality~\eqref{eq:bapat-sunder} holds with $A=B=L_G$.

It remains open whether Chollet's inequality holds for every graph Laplacian, or even whether $\per(L_G\circ L_G)\le\per(L_G)\prod_{v\in V(G)}\deg_G(v)$.

\begin{credits}
\subsubsection{\ackname}
Priyanshu Pant is supported by the CSIR--UGC NET Junior Research Fellowship. This work is also supported by ANRF Grant ANRF/\allowbreak ARGM/\allowbreak 2025/\allowbreak 002341/\allowbreak TS.
\end{credits}

\bibliographystyle{splncs04}
\bibliography{references}

\clearpage
\appendix

\section{Proof of Hereditary Strong Closure}
\label{app:strong-closure}

\begin{proof}[Theorem~\ref{thm:hereditary-strong}]
Write
\[
A_i=
\begin{pmatrix}
B_i&c_i\\
c_i^\top&d_i
\end{pmatrix}
\qquad\text{and}\qquad
h_i=\prod_j(B_i)_{jj}.
\]
Since $A_i$ and $B_i$ satisfy the strong inequality,
\[
\per(A_i\circ A_i)\le d_i h_i\per(A_i),
\qquad
\per(B_i\circ B_i)\le h_i\per(B_i).
\]
Proposition~\ref{prop:one-point-formulas} therefore gives
\[
\begin{aligned}
\per\bigl((A_1\star A_2)\circ(A_1\star A_2)\bigr)
\le h_1h_2\bigl(
&d_1\per(A_1)\per(B_2)
+d_2\per(B_1)\per(A_2)\\
&+2d_1d_2\per(B_1)\per(B_2)
\bigr).
\end{aligned}
\]
On the other hand,
\[
\per(A_1\star A_2)\prod_j(A_1\star A_2)_{jj}
=
h_1h_2(d_1+d_2)
\bigl(\per(A_1)\per(B_2)+\per(B_1)\per(A_2)\bigr).
\]
The difference between the right-hand sides is
\[
\begin{aligned}
h_1h_2\Bigl[
&d_2\per(B_2)\bigl(\per(A_1)-d_1\per(B_1)\bigr)\\
&+
d_1\per(B_1)\bigl(\per(A_2)-d_2\per(B_2)\bigr)
\Bigr],
\end{aligned}
\]
which is nonnegative by Corollary~\ref{cor:lieb-principal}.

Now let $C$ be a principal submatrix of $A_1\star A_2$. If $C$ omits
the common index and lies entirely on one side, the result follows from
the hereditary strong property of $A_1$ or $A_2$. Otherwise
$C=C_1\oplus C_2$, and multiplicativity of the permanent gives the
strong inequality for $C$.

Suppose next that $C$ contains the common index and vertices from both
sides. Then $C$ is a one-point sum of principal submatrices of $A_1$
and $A_2$, and the argument above applies to them.

It remains to consider the case in which $C$ contains the common index
and vertices from only one side. If $C$ consists only of the common
index, the result is immediate. Otherwise write
$C=D+\alpha E_{vv}$, where $\alpha\ge0$ and $D=(d_{ij})$ is a principal
submatrix of $A_1$ or $A_2$ containing $v$. Put
$d=d_{vv}$ and $h=\prod_{u\ne v}d_{uu}$. Since $D$ and $D(v)$ satisfy
the strong inequality,
\[
\begin{aligned}
\per(C\circ C)
&=
\per(D\circ D)
+
(2d\alpha+\alpha^2)\per(D(v)\circ D(v))\\
&\le
h\bigl(d\per(D)+(2d\alpha+\alpha^2)\per(D(v))\bigr).
\end{aligned}
\]
By Corollary~\ref{cor:lieb-principal},
$\per(D)\ge d\per(D(v))$, and therefore
\[
d\per(D)+(2d\alpha+\alpha^2)\per(D(v))
\le
(d+\alpha)\bigl(\per(D)+\alpha\per(D(v))\bigr).
\]
Since $\per(C)=\per(D)+\alpha\per(D(v))$ and
$\prod_u c_{uu}=(d+\alpha)h$, the result follows.
\end{proof}

\section{Complete-graph calculations}
\label{app:clique-calculations}

We give the two calculations used in the proof of
Theorem~\ref{thm:complete-graphs}.

\begin{lemma}\label{lem:nIminusJ}
Let $n\ge2$, $s\ge1$, and $M=nI_s-J_s$. Then
\[
\per(M)
=
(-1)^s s!\sum_{r=0}^{s}\frac{(-n)^r}{r!},
\qquad
\per(M\circ M)
=
s!\sum_{r=0}^{s}\frac{(n(n-2))^r}{r!}.
\]
\end{lemma}

\begin{proof}
The diagonal entries of $M$ are $n-1$ and the off-diagonal entries are
$-1$. For $\sigma\in S_s$, let
$f(\sigma)=|\{i:\sigma(i)=i\}|$. Then
\[
\prod_i m_{i,\sigma(i)}
=
(n-1)^{f(\sigma)}(-1)^{s-f(\sigma)}
=
(-1)^s(1-n)^{f(\sigma)}.
\]
Expanding $(1-n)^{f(\sigma)}$ gives
\[
\per(M)
=
(-1)^s
\sum_{r=0}^{s}
(-n)^r
\sum_{\sigma\in S_s}\binom{f(\sigma)}r.
\]
The inner sum counts a permutation together with $r$ chosen fixed
points. Hence
\[
\sum_{\sigma\in S_s}\binom{f(\sigma)}r
=
\binom sr(s-r)!
=
\frac{s!}{r!},
\]
which proves the first formula.
For $M\circ M$, the off-diagonal entries are $1$ and the diagonal
entries are $(n-1)^2=1+n(n-2)$. The same argument, with $n(n-2)$ in
place of $-n$, gives the second formula.
\end{proof}

\begin{lemma}\label{lem:clique-scalar}
For every $n\ge2$,
\[
n!\sum_{k=0}^{n}\frac{(n(n-2))^k}{k!}
\le
\left(
n!\sum_{k=0}^{n}\frac{(-n)^k}{k!}
\right)^2.
\]
\end{lemma}

\begin{proof}
For $2\le n\le7$, Lemma~\ref{lem:nIminusJ} gives
\[
\begin{array}{c|r|r|r}
n & \per(L_{K_n})
  & \per(L_{K_n}\circ L_{K_n})
  & \Delta(K_n)\\
\hline
2 & 2 & 2 & 2\\
3 & 12 & 78 & 66\\
4 & 120 & 7128 & 7272\\
5 & 1480 & 1095420 & 1094980\\
6 & 22320 & 250716240 & 247466160\\
7 & 396564 & 79767489690 & 77495516406
\end{array}
\]
so the inequality holds.
Assume $n\ge8$ and put
$a_k=(n(n-2))^k/k!$. For $1\le k\le n$,
$a_{k-1}/a_k=k/(n(n-2))\le1/(n-2)$. Hence
\[
\sum_{k=0}^{n}a_k
\le
\frac{n-2}{n-3}
\frac{(n(n-2))^n}{n!}.
\]
Now put $T_n=\sum_{k=0}^{n}(-n)^k/k!$. The last two terms cancel, so
$T_n=\sum_{k=0}^{n-2}(-n)^k/k!$. Taylor's theorem with integral
remainder gives
\[
e^{-n}
=
T_n+
\frac{(-1)^{n-1}}{(n-2)!}
\int_0^n e^{-u}(n-u)^{n-2}\,du.
\]
With $u=nt$ and $e^{-nt}\ge(1-t)^n$ for $0\le t\le1$,
\[
\begin{aligned}
\int_0^n e^{-u}(n-u)^{n-2}\,du
&=
n^{n-1}\int_0^1e^{-nt}(1-t)^{n-2}\,dt\\
&\ge
n^{n-1}\int_0^1(1-t)^{2n-2}\,dt
=
\frac{n^{n-1}}{2n-1}.
\end{aligned}
\]
Since $e^{-n}\le1/(4n)$ for $n\ge4$,
\[
|T_n|
\ge
\frac{n^{n-1}}{(2n-1)(n-2)!}
-
\frac1{4n}.
\]
The difference between the first term and
$n^{n-2}/(2(n-2)!)$ is
$n^{n-2}/(2(2n-1)(n-2)!)$, which is at least $1/(4n)$.
Therefore
\[
|T_n|
\ge
\frac{n^{n-2}}{2(n-2)!}.
\]
It remains to prove
\[
\frac{n-2}{n-3}\frac{(n(n-2))^n}{n!}
\le
n!\left(\frac{n^{n-2}}{2(n-2)!}\right)^2.
\]
Using $n!=n(n-1)(n-2)!$ and cancelling common factors, this is
equivalent to
\[
\left(\frac{n}{n-2}\right)^{n-2}
\ge
\frac{4(n-2)^3}{(n-3)(n-1)^2}.
\]
By the binomial theorem,
\[
\left(1+\frac2{n-2}\right)^{n-2}
\ge
1+2+\binom{n-2}{2}\left(\frac2{n-2}\right)^2
=
5-\frac2{n-2}
\ge4.
\]
Also,
$(n-1)^2(n-3)-(n-2)^3=n^2-5n+5\ge0$ for $n\ge8$, so
\[
\frac{4(n-2)^3}{(n-3)(n-1)^2}\le4.
\]
This proves the result.
\end{proof}

\end{document}